\documentclass[11pt]{article}
\usepackage[utf8]{inputenc}
\usepackage[margin=1in]{geometry}
\usepackage[english]{babel}
\usepackage{subcaption}
\usepackage{amsmath}
\usepackage{amssymb}
\usepackage{amsfonts}
\usepackage{amsthm}
\usepackage{mathtools}
\usepackage{soul}
\usepackage{mathdots}
\usepackage{mathrsfs}

\usepackage{hyperref}
\usepackage[
style=numeric, 
sorting=ynt,
backend = biber,
doi = false,
url = false,
]{biblatex}
\AtEveryBibitem{
\ifentrytype{misc}{}{\clearfield{eprint}}
}
\usepackage{csquotes}
\usepackage[all]{xy}
\usepackage[pdftex]{graphicx}
\usepackage{color}
\usepackage{url}
\usepackage{enumitem}

\theoremstyle{plain}
\newtheorem{thm}{Theorem}
\newtheorem{lemma}[thm]{Lemma}

\newtheorem{qn}[thm]{Question}

\theoremstyle{definition}
\newtheorem{defn}[thm]{Definition}

\theoremstyle{remark}

\newtheorem*{prop*}{Proposition}
\newtheorem*{claim*}{Claim}

\numberwithin{equation}{section}
\numberwithin{thm}{section}
\makeatletter
\newcommand{\tpmod}[1]{{\@displayfalse\pmod{#1}}}
\makeatother

\title{Irreducible Polynomials with Coefficients in an Affine Algebraic Set}
\author{Neil Kolekar\footnote{\href{mailto:nkolekar@gmail.com}{nkolekar@gmail.com}}}
\date{\today}

\begin{document}
\setlength{\parindent}{15pt}

\maketitle
\begin{abstract}
    In this paper, we give error bounds on the number of monic irreducible polynomials $a_0+a_1x+\dots+a_{n-1}x^{n-1}+x^n$ over a finite field $\mathbb{F}_q$ of degree $n$ with $(a_0, a_1, \dots, a_{n-1}, 1)$ lying in a fixed affine algebraic set $V$ of points in $\mathbb{F}_q^{n+1}$. 
\end{abstract}

\section{Introduction}
The study of irreducible polynomials with prescribed coefficients over finite fields carries a rich history, with much progress having been made over the last three decades. While the problem of counting irreducible polynomials over finite fields can be traced back to the early 19th century when Gauss proved that there are approximately $q^n/n$ irreducible polynomials of degree $n$ in $\mathbb{F}_q[x]$, this problem with the additional restriction of prescribed coefficients was called into question in 1992 by Hansen and Mullen \cite{Hansen92}, who conjectured that for all prime powers $q \ge 3$ and positive integers $n$, one can find a monic irreducible polynomial of degree $n$ in $\mathbb{F}_q[x]$ with any single coefficient prescribed. In 1997, Wan \cite{Wan97} proved the Hansen-Mullen conjecture for $n \ge 36$ or $q > 19$, leaving finitely many cases remaining; the conjecture was proved in its entirety by Ham and Mullen, who computationally verified the remaining cases. 

In general, error bounds on the number of irreducible polynomials with prescribed coefficients have played a key role in the literature of the subject. Their practical usage is mostly present in polynomial existence proving, as proving existence analytically is often much easier compared to direct approaches. However, their ability to represent bias towards some coefficient selections over others is also of interest; in the broader context of arithmetic geometry, for instance, the \textit{analytic rank} is defined to measure the bias of tensors, and has been a topic of interest to many researchers in the field \cite{moshkovitz2024quasilinearrelationpartitionanalytic}. 

Since it is hard to understand why singular coefficient prescriptions differ in representation among irreducible polynomials, we are motivated to consider counting irreducible polynomials $a_0+a_1x+\dots+a_{n-1}x^{n-1}+x^n \in \mathbb{F}_q[x]$ whose coefficient vector $(a_0, \dots, a_{n-1}, 1)$ lies in a fixed affine set $V$ of points in $\mathbb{F}_q^{n+1}$, which is the objective of this paper. Indeed, such restrictions have been studied in depth in the context of positive integers and prescribed digits, primarily using tools from analytic number theory. Nevertheless, it is notable that this has also been studied briefly in the finite fields setting, which is often easier to work in than the positive integers; see \cite{Gao2021} for self-reciprocal irreducible polynomials, for example. An important work in this line of research is that of Mérai, who gave error bounds \cite{Mérai2025} for the distribution of the Rudin-Shapiro function over irreducible polynomials in $\mathbb{F}_q[x]$. This was one of the first error bounds obtained on the distribution of a polynomial function in the coefficients of irreducible polynomials over finite fields.

The goal of this paper is as follows. 
\begin{quote}
\textsl{Fix polynomials $R_1, \dots, R_m$ on the coefficients of monic polynomials in $\mathbb{F}_q[x]$ of degree $n$. Find error bounds for the number $I_n(R_1, \dots, R_m)$ of monic irreducible polynomials $f$ of degree $n$ such that $(R_1(f), \dots, R_m(f)) = (0, \dots, 0)$.}
\end{quote}
Since on average, the polynomials $R_1$, \dots, $R_m$ take on equidistributed values, we should expect to have \[ I_n(R_1, \dots, R_m) \approx \frac{I(n)}{q^m}, \] where $I(n) = q^n/n + O(q^{n/2}/n)$ is the number of irreducible polynomials of degree $n$ over $\mathbb{F}_q$. Thus, we provide an upper bound on the difference $\left|I_n(R_1, \dots, R_m) - \frac{I(n)}{q^m}\right|$. 
\begin{thm}\label{main theorem}
  Let $\mathbb{F}_q$ be a finite field of characteristic $p$ with $q$ odd. Suppose that $n < p$. Let $j:=\max(\deg R_1, \dots, \deg R_m)$, and let $S$ denote the set of indices $1 \le i \le m$ for which $\deg R_i(f) = j$. Define $(P_i)_g(f)=R_i(fg)$ and $(P_i)_{g_1, g_2}(f) = R_i(g_1f) - R_i(g_2f)$. There exists a constant $c>0$, dependent only on $j$, such that \[ \left|I_n(R_1, \dots, R_m)-\frac{I(n)}{q^m}\right| \le \frac{q^m-1}{q^m} \cdot \min_{0 \le u+v < n} \left(nS_1(u, v) + n^{5/2}q^{n-(u+v)/2}\sqrt{S_2(u, v)}\right) \] where \[ S_1(u, v) = \sum_{0 \le d \le u+v} \sum_{\deg g = d} q^{(n-d)+\frac{n-d}{2^j}-c \cdot \operatorname{rank}(((P_i)_g)_{i \in S})} \] and \[ S_2(u, v) = \max_{v \le k \le n-u} \max_{\deg g_1 = n-k} \sum_{\deg g_2 = n-k} q^{k-c \cdot \operatorname{rank}(((P_i)_{g_1, g_2})_{i \in S})}. \]
\end{thm}

The structure of this paper is as follows. We begin by stating the preliminaries assumed throughout this paper in Section \ref{prelim}. Next, in Section \ref{lemmas}, we state and prove the lemmas that play a key role in the proof of our main result. Our approach, based on character sums, uses a function field analogue of Vaughan's identity, polarization to reduce the problem to working with multilinear exponential sums, and partition rank to bound the multilinear exponential sums. In Section \ref{main proof}, we prove the main result by applying the lemmas from Section \ref{lemmas}, and we end with Section \ref{future directions}, in which we discuss future research directions.   

\section{Preliminaries}\label{prelim}
Throughout this paper, we will let $q$ be an odd prime power with prime divisor $p$. We will use $\mathcal{M}(n)$ to denote the set of polynomials of degree $n$ in $\mathbb{F}_q[x]$, and $\mathcal{P}(n)$ to denote the set of irreducible polynomials in $\mathcal{M}(n)$. Also, let $R_1(f), \dots, R_m(f)$ denote $m$ polynomial functions on the coefficients of $f$. Denote $\Lambda$ by the $\mathbb{F}_q[x]$-analogue of the \emph{von Mangoldt function}, that is, \[ \Lambda(f) = \begin{cases} \deg P & f=P^r, \text{where $P$ is irreducible} \\ 0 & \text{otherwise.} \end{cases}. \] Use $\psi$ (often with a subscript) to denote an additive character on $\mathbb{F}_q$. It is well-known that there is a bijection $\varphi$ mapping the group $\widehat{\mathbb{F}_q}$ of additive characters over $\mathbb{F}_q$ to $\mathbb{F}_q$ such that \[ \psi(x) = e^{2\pi i \operatorname{Tr}(\varphi(\psi)x)} \] for all $\psi \in \widehat{\mathbb{F}_q}$ and $x \in \mathbb{F}_q$. Let $\psi_0$ denote the additive character with $\varphi(\psi_0) = 1$. The final important terminology we introduce is that of the rank of polynomials, and more generally, the rank of tuples of polynomials. The \emph{rank} $\operatorname{rank}(f)$ of a multivariable polynomial $f$ is defined as the smallest number $r$ such the degree $\deg(f)$ homogeneous component of $f$ can be expressed as the sum of $r$ reducible homogeneous polynomials. Furthermore, the \emph{rank} $\operatorname{rank}(R_1, \dots, R_m)$ of a tuple $(R_1, \dots, R_m)$ of polynomials is given by the minimum rank of any nontrivial linear combination of $R_1, \dots, R_m$. 

\section{Auxiliary lemmas}\label{lemmas}
\subsection{Vaughan's identity in $\mathbb{F}_q[x]$}
In \cite{Mérai2025}, the following variant of Vaughan's identity for arithmetic functions on $\mathbb{N}$ is utilized. This bound encapsulates the presence of the von Mangoldt function in the character sums that we work with later.

\begin{lemma}\label{Vaughan's identity}
  Let $\Psi: \mathbb{F}_q[x] \to \mathbb{C}$ be a function with $|\Psi(f)| \le 1$. If $u$ and $v$ are two integers in $[1, n]$ with $u+v>n$, then \[ \left|\sum_{f \in \mathcal{M}(n)} \Lambda(f)\Psi(f)\right| \ll n\Sigma_1 + n^{5/2}q^{n-(u+v)/2}\Sigma_2^{1/2}, \] with \[ \Sigma_1 = \sum_{\deg g \le u+v} \left|\sum_{\deg h = n - \deg g} \Psi(gh)\right| \] and \[ \Sigma_2 = \max_{v \le i \le n-u} \max_{\deg g_1 = n-i} \sum_{\deg g_2 = n-i} \left|\sum_{\deg h = i} \Psi(hg_1)\overline{\Psi(hg_2)}\right|, \] where both sums above are taken over monic polynomials.
\end{lemma}
In particular, it suffices to obtain upper bounds on $\Sigma_1$ and $\Sigma_2$. Using polarization, we will reduce the computation to computing character sums on multilinear polynomials. 

\subsection{Polarization}
In this subsection, we introduce the technique of polarization. Given a function $R(f)$ of the coefficients of $f$ and polynomial $h$ with $\deg h < \deg f$, let $\Delta_hR(f) := R(f+h)-R(f)$. The key idea is that \[ \Delta_{h_1} \cdots \Delta_{h_d} R(f) \] is a multilinear polynomial in $h_1, \dots, h_d$ which is independent of $f$, called the \emph{polarization} of $R(f)$ and denoted $R^{\circ}(h_1, \dots, h_d)$. In particular, $R^{\circ}$ is a $d$-tensor and satisfies $R^{\circ}(x, \dots, x) = R(x)$. Polarization has been used in several instances for the computation of exponential sums; please refer to \cite{Schmidt1984}, \cite{Gowers_2011}, and \cite{moshkovitz2024quasilinearrelationpartitionanalytic} for examples.

\begin{lemma}\label{Polarization result}
  Let $P_1(f), \dots, P_m(f)$ be polynomials in the coefficients of a $k$-degree polynomial $f$. Let $j=\max(\deg P_1, \dots, \deg P_m)$ and $S$ denote the set of indices $i \in [1, m]$ for which $\deg P_i = j$. Then we have \[ \left|\sum_{\deg f = k} \prod_{i=1}^m \psi_i(P_i(f))\right| \le \left(q^{k}\left|\sum_{x \in (\mathbb{F}_q^k)^j} \prod_{i \in S} \psi_i(P_i^{\circ}(x)) \right|\right)^{1/2^j},\] where the sum above is over monic $f$. 
\end{lemma}
\begin{proof}
  Note that \begin{align*}
    \left|\sum_{\deg f = k} \prod_{i=1}^m \psi_i(P_i(f))\right|^2 & = \sum_{\deg f = k} \sum_{\deg h < k} \prod_{i=1}^m \psi_i(P_i(f+h))\overline{\psi_i(P_i(f))} \\ &= \left|\sum_{\deg f = k} \sum_{\deg h < k} \prod_{i=1}^m \psi_i(\Delta_h P_i(f))\right| \\ &\le \sum_{\deg f = k}\left| \sum_{\deg h < k} \prod_{i=1}^m \psi_i(\Delta_h P_i(f))\right|. 
  \end{align*}
  Repeatedly squaring in this fashion, we obtain that at the $j$th squaring, \[ \left|\sum_{\deg f = k} \prod_{i=1}^m \psi_i(P_i(f))\right|^{2^j} \le q^{k} \left|\sum_{h_1, \dots, h_j} \prod_{i \in S} \psi_i(P_i^{\circ}(h_1, \dots, h_j)) \right|, \] as desired. 
\end{proof}
Recall that the purpose of this application of polarization is to obtain upper bounds on the sums $\Sigma_1$ and $\Sigma_2$ from Lemma \ref{Vaughan's identity} upon selecting $\Psi(f) = \prod_{i=1}^m \psi_i(R_i(f))$. Indeed, we may consider the following choices of $P_i$, which are both polynomials in the coefficients of their arguments. 
\begin{itemize}
  \item For $\Sigma_1$, let $(P_i)_g(f)=R_i(fg)$ for all $f \in \mathcal{M}(n-\deg g)$ and $1 \le i \le m$. 
  \item For $\Sigma_2$, let $(P_i)_{g_1, g_2}(f) = R_i(g_1f) - R_i(g_2f)$ for all $f \in \mathcal{M}(n - \max(\deg g_1, \deg g_2))$ and $1 \le i \le m$. 
\end{itemize}

\subsection{Exponential sums of multilinear polynomials}
Note that Lemma \ref{Polarization result} requires us to only work with exponential sums of multilinear polynomials going forward. In order to obtain useful bounds on the exponential sums of multilinear polynomials, we require a corollary of a recent result on the comparison between the analytic rank and partition rank of tensors.
\begin{defn}
  The \emph{analytic rank} $\operatorname{AR}(T)$ of $T$ is defined as \[ \operatorname{AR}(T) := -\log_q\left(\frac{1}{q^n} \sum_{(x_1, \dots, x_n) \in \mathbb{F}_q^n} \psi_0(T(x_1, \dots, x_n))\right). \] The \emph{partition rank} $\operatorname{PR}(T)$ of $T$ is defined as the smallest number $r$ such that $T$ can be written as the sum of $r$ reducible homogeneous polynomials. 
\end{defn}

\begin{lemma}[Moshkovitz-Zhu, \cite{moshkovitz2024quasilinearrelationpartitionanalytic}]
  There exists a constant $c > 0$, dependent only on $j$, such that for all $j$-tensors $T$, \[ \operatorname{AR}(T) \ge c \cdot \operatorname{PR}(T). \]
\end{lemma}

We will now state a technical lemma regarding polarizations and ranks along with a corollary pertaining to character sums. Combined with Lemma \ref{Vaughan's identity}, this lemma eliminates the need to consider character sum bounds beyond this. 

\begin{lemma}\label{char sums and ranks}
For all polynomials $f$ of degree $n<p$, $\operatorname{PR}(f^{\circ}) \ge \operatorname{rank}(f)$. Consequently, for any tuple $(P_1, \dots, P_m)$ of polynomials and nontrivial tuple $(\psi_1, \dots, \psi_m)$ of additive characters on $\mathbb{F}_q$, there exists a constant $c>0$ dependent only on $\max(\deg P_1, \dots, \deg P_m)$ such that \[ \left|\sum_{x \in (\mathbb{F}_q^k)^j} \prod_{i \in S} \psi_i(P_i^{\circ}(x)) \right| \le q^{n-c \cdot \operatorname{rank}((P_i)_{i \in S})}, \] where $j = \max(\deg P_1, \dots, \deg P_m)<p$ and $S$ is the set of indices $1 \le i \le m$ for which $\deg P_i = j$. 
\end{lemma}
\begin{proof}
Note that because $n<p$, the degree $n$ homogeneous component of $f$ is given by $\frac{1}{n!}f^{\circ}(x, \dots x)$. Thus, \[ \operatorname{rank}(f) = \operatorname{rank}(f^{\circ}(x, \dots x)) \le \operatorname{PR}(f^{\circ}), \] proving the first part. Let $L$ denote the set of nontrivial $\mathbb{F}_q$-combinations of $\{P_i: i \in S\}$. Observe that \begin{align*} \left|\sum_{x \in (\mathbb{F}_q^k)^j} \prod_{i \in S} \psi_i(P_i^{\circ}(x)) \right| = \left|\sum_{x \in (\mathbb{F}_q^k)^j} \psi_0\left(\sum_{i \in S} \varphi(\psi_i)P_i^{\circ}(x)\right) \right| \le \max_{P \in L} \left|\sum_{x \in (\mathbb{F}_q^k)^j} \psi_0\left(P^{\circ}(x)\right) \right| \end{align*} and \begin{align*} \max_{P \in L} \left|\sum_{x \in (\mathbb{F}_q^k)^j} \psi_0\left(P^{\circ}(x)\right) \right| &\le \max_{P \in L} q^{n-\operatorname{AR}(P^{\circ})} \\ &\le \max_{P \in L} q^{n-c \cdot \operatorname{PR}(P^{\circ})} \\ &\le \max_{P \in L} q^{n-c \cdot \operatorname{rank}(P)} \\ &= q^{n-c \cdot \operatorname{rank}((P_i)_{i \in S})}, \end{align*} which proves the second part.
\end{proof}

\section{Proof of main result}\label{main proof}
We are ready to prove our main result.
\begin{proof}[Proof of Theorem \ref{main theorem}]
By the orthogonality of characters, we have \[ I_n(R_1, \dots, R_m) = \frac{1}{q^m} \sum_{(\psi_1, \dots, \psi_m)} \sum_{f \in \mathcal{P}(n)} \prod_{i=1}^m \psi_i(R_i(f)). \] Thus, \[ I_n(R_1, \dots, R_m) - \frac{I(n)}{q^m} = \frac{1}{q^m} \sum_{(\psi_1, \dots, \psi_m) \neq \mathbf{0}} \sum_{f \in \mathcal{P}(n)} \prod_{i=1}^m \psi_i(R_i(f)). \] We may rewrite the above as \begin{align*} \sum_{f \in \mathcal{P}(n)} \prod_{i=1}^m \psi_i(R_i(f)) &= \sum_{f \in \mathcal{M}(n)} \mathbf{1}_{f \in \mathcal{P}(n)}\prod_{i=1}^m \psi_i(R_i(f)) \\ &= \sum_{f \in \mathcal{M}(n)} \Lambda(f) \prod_{i=1}^m \psi_i(R_i(f)) + O\left(\sum_{f \in \mathcal{M}(n)} \mathbf{1}_{\Lambda(f) > 1}\right) \\ &= \sum_{f \in \mathcal{M}(n)} \Lambda(f) \prod_{i=1}^m \psi_i(R_i(f)) + O(I(n)-I(n/2))  \\ &= \sum_{f \in \mathcal{M}(n)} \Lambda(f) \prod_{i=1}^m \psi_i(R_i(f)) + O(q^{n/2}/n) \end{align*} Therefore, \[ \left|I_n(R_1, \dots, R_m) - \frac{I(n)}{q^m}\right| \le \sum_{(\psi_1, \dots, \psi_m) \neq \mathbf{0}} \left|\sum_{f \in \mathcal{M}(n)} \Lambda(f)\prod_{i=1}^{m}  \psi_i(R_i(f))\right| \] by the triangle inequality. By Lemma \ref{Vaughan's identity}, we have for all nontrivial tuples $(\psi_1, \dots, \psi_m)$ of additive characters on $\mathbb{F}_q$ that
\[ \left|\sum_{f \in \mathcal{M}(n)} \Lambda(f)\prod_{i=1}^m \psi_i(R_i(f))\right| \ll n\Sigma_1 + n^{5/2}q^{n-(u+v)/2}\Sigma_2^{1/2} \] 
for all choices of $u$ and $v$ with $u+v<n$, where 
\[ \Sigma_1 = \sum_{\deg g \le u+v} \left|\sum_{\deg h = n - \deg g} \prod_{i=1}^m \psi_i((P_i)_g(h)) \right| \] 
and 
\[ \Sigma_2 = \max_{v \le k \le n-u} \max_{\deg g_1 = n-k} \sum_{\deg g_2 = n-k} \left|\sum_{\deg h = k} \prod_{i=1}^m \psi_i((P_i)_{g_1, g_2}(h)) \right|. \] 
By Lemma \ref{Polarization result}, we have 
\[ \Sigma_1 \le \sum_{0 \le d \le u+v} q^{\frac{n-d}{2^j}} \sum_{\deg g = d} \left|\sum_{x \in (\mathbb{F}_q^d)^{n-d}} \prod_{i \in S} \psi_i((P_i)_g^{\circ}(x))\right|^{1/2^j} \] 
and
\[ \Sigma_2 = \max_{v \le k \le n-u} \max_{\deg g_1 = n-k} \sum_{\deg g_2 = n-k} q^{\frac{k}{2^j}}\left|\sum_{x \in (\mathbb{F}_q^{n-k})^k} \prod_{i \in S} \psi_i((P_i)_{g_1, g_2}^{\circ}(x))\right|^{1/2^j}. \] 
Thus, by Lemma \ref{char sums and ranks}, we have 
\[ \Sigma_1 = \sum_{0 \le d \le u+v} \sum_{\deg g = d} q^{(n-d)+\frac{n-d}{2^j}-c \cdot \operatorname{rank}(((P_i)_g)_{i \in S})} \] 
and 
\[ \Sigma_2 = \max_{v \le k \le n-u} \max_{\deg g_1 = n-k} \sum_{\deg g_2 = n-k} q^{k-c \cdot \operatorname{rank}(((P_i)_{g_1, g_2})_{i \in S})}. \] The desired result follows. 
\end{proof}

\section{Future Directions}\label{future directions}
In this section, we provide several research questions that allow for potential extensions of this work. 

First, we note that our error bounds rely on the assumption that $n < p$; that is, we work over large finite fields. This limitation lies in the proof of Lemma \ref{char sums and ranks}, which establishes a bound on the partition rank of the polarization of polynomials. 

\begin{qn}
  Is there a way to prove a (potentially weaker) version of Lemma \ref{char sums and ranks}? Can this be rectified for $p \ge n$ by working $p$-adically?
\end{qn}

Observe that our error bounds are independent of the polynomials in $\{R_1, \dots, R_m\}$ whose degree is not $\max(\deg R_1, \dots, \deg R_m)$. 
\begin{qn}
  Can our method of carrying out polarization be improved to obtain refined bounds that take into account of \emph{all} $R_i$? 
\end{qn}
Currently, there are no rank inequalities available that allow us to simplify the ranks in our error bounds. 
\begin{qn}
  Find upper bounds for $\operatorname{rank}(((P_i)_g)_{i \in S})$ and $\operatorname{rank}(((P_i)_{g_1, g_2})_{i \in S})$ in terms of $\operatorname{rank}(R_i)$, $g$, $g_1$, and $g_2$. 
\end{qn}

\section*{Acknowledgements}
I would like to thank Dr.\ Simon Rubinstein-Salzedo for useful discussions during the creation of this paper. 

\printbibliography

\end{document}